\documentclass[a4paper]{article}

\usepackage[english]{babel}
\usepackage[utf8]{inputenc}
\usepackage{amsmath}
\usepackage{amsthm}
\usepackage{amssymb}
\usepackage{graphicx}
\usepackage[colorinlistoftodos]{todonotes}
\usepackage{hyperref} 
\usepackage{color}
\usepackage{verbatim}

\newtheorem{theorem}{Theorem}
\newtheorem{lemma}[theorem]{Lemma}
\newtheorem{corollary}[theorem]{Corollary}

\theoremstyle{definition}
\newtheorem{definition}{Definition}
\newtheorem{example}{Example}
\newtheorem{strategy}{Strategy}
\newtheorem{innercustomstr}{Strategy}
\newenvironment{customstr}[1]
  {\renewcommand\theinnercustomstr{#1}\innercustomstr}
  {\endinnercustomstr}

\title{Weighing Coins and Keeping Secrets}
\author{Nicholas Diaco \and Tanya Khovanova}

\begin{document}
\maketitle

\begin{abstract}
In this expository paper we discuss a relatively new counterfeit coin problem with an unusual goal: maintaining the privacy of, rather than revealing, counterfeit coins in a set of both fake and real coins. We introduce two classes of solutions to this problem --- one that respects the privacy of all the coins and one that respects the privacy of only the fake coins --- and give several results regarding each. We describe and generalize 6 unique strategies that fall into these two categories. Furthermore, we explain conditions for the existence of a solution, as well as showing proof of a solution's optimality in select cases. In order to quantify exactly how much information is revealed by a given solution, we also define the revealing factor and revealing coefficient; these two values additionally act as a means of comparing the relative effectiveness of different solutions. Most importantly, by introducing an array of new concepts, we lay the foundation for future analysis of this very interesting problem, as well as many other problems related to privacy and the transfer of information.
\end{abstract}

\section{Introduction}

In 2007 Alexander Shapovalov suggested an unusual coin weighing problem for the sixth international Kolmogorov math tournament \cite{TurKolm}.

\begin{quote}
A judge is presented with 80 coins that all look the same, knowing that there are either two or three fake coins among them. All the real coins weigh the same and all the fake coins weigh the same, but the fake coins are lighter than the real ones.

A lawyer knows that there are exactly three fake coins and which ones they are. The lawyer must use a balance scale to convince the judge that there are exactly three fake coins. She is bound by her contract to not reveal any information about any particular coin. How should she proceed?
\end{quote}

Why is this problem so unusual? Let's take a look back at history. The first coin weighing problems appeared around 1945 \cite{Newbery, Schwartz}. In all of them, the goal was simply to find a single fake coin amongst many real coins. After that, many generalizations followed: newer versions of the counterfeit coin problem included finding multiple fake coins, differentiating between coins of arbitrary weights, and so on. All of them, however, had the additional goal of minimizing the number of weighings necessary to locate the fake coin(s); see \cite{GN} and its many references.

Shapovalov's puzzle is the first problem to switch the attention to maintaining the privacy of coins rather than eliminating it. This puzzle is very important and modern; like many other ``coin weighing'' problems, it is not about coins---rather, it uses coins to model ideas and create a simplified version of real life privacy problems and their potential solutions.

\section{A Simplified Version}\label{simp}

Let us consider a simpler version of Shapovalov's puzzle to get our feet wet:

\begin{example}\label{ex:80-2-1}  

A lawyer presents 80 identical coins to a judge, who knows that among them there are either one or two fake coins. All the real coins weigh the same and all the fake coins weigh the same, but the fake coins are lighter than the real ones.

The lawyer knows that there are exactly two fake coins and which ones they are. Can the lawyer use a balance scale to convince the judge that there are exactly two fake coins without revealing any information about any particular coin?
\end{example}

We will offer several strategies for different versions of this puzzle, so we would like to number them. Here is the first strategy to solve Example~\ref{ex:80-2-1}:

\begin{strategy}\label{str:1}
The lawyer divides the coins into two piles of 40 so that each pile contains exactly one fake coin. She then puts the piles in the different pans of the scale.
\end{strategy}

The scale will balance, which means that the number of fake coins is the same in both of the pans. Therefore, the total number of fake coins is even, and in this case is exactly 2. For any particular coin, the judge can't definitively say whether it is real or fake; we have thus succeeded in our task.

Let us introduce some notation before we move forward. We will denote the \textbf{t}otal number of coins by $t$, the actual number of \textbf{f}ake coins by $f$, and the number of fake coins that we are trying to \textbf{d}isprove by $d$.

We would also like to give a name to a strategy or a series of weighings after which the judge knows nothing about the authenticity of any specific coin. Knop \cite{Knop} suggests that such strategies be called \textit{shapovalov} in honor of the puzzle's designer. One of the authors \cite{TK} uses the name \textit{unrevealing}. We do not like the name ``unrevealing'' as we plan to show that all strategies do reveal some information. We like the name ``shapovalov,'' but we also want to have a descriptive name.

\begin{definition}
We will call a set of weighings or a strategy where no information about any particular coin is revealed \textit{discreet}. Otherwise, we call the set of weighings \textit{indiscreet}.
\end{definition}

For the time being, we are only concerned with discreet (shapovalov) strategies. We will refer to a given example with the set of three numbers $t$-$f$-$d$. For example, we will refer to Example~\ref{ex:80-2-1} as Example~\ref{ex:80-2-1} (80-2-1).

\begin{customstr}{1*}\label{str:1*}
The lawyer divides the coins into $f$ equal groups with one fake coin in each and shows that all of them are equal in weight.
\end{customstr}

Later, we will see that this strategy can be adjusted to a more general case when $f$ and $t$ have a common divisor that doesn't divide $d$.

Now we can come back to the original puzzle (80-3-2) and discuss its solution.

\section{Solutions to the Original Problem}

Motivated by the strategy used in the previous example, the lawyer can try to divide 80 coins into three groups, each containing one fake coin. Clearly, 80 is not divisible by 3, so she makes the three largest possible groups of the same size: each of 26 elements with one fake coin. The lawyer uses two weighings on the scale to demonstrate to the judge that all three groups have the same weight. What can the judge conclude? He can conclude that either there are 3 fake coins---one in each group, or there are 2 fake coins and they are in the leftover group of 2 coins. Unfortunately, this strategy is not good enough to prove to the judge that there are exactly 3 fake coins. The lawyer decides not to give up and adjusts the strategy in the following manner:

\begin{strategy}\label{str:2} (80-3-2)

The lawyer starts by showing that the three groups of 26 coins, containing one fake coin each, have the same weight. She continues by comparing one of the coins in the leftover group to a real coin not in the leftover group.
\end{strategy}

In this case, the judge can deduce that one of the leftover coins can balance against a coin from one of the larger groups only when the leftover group does not contain fake coins. Therefore, there must be 3 fake coins. The lawyer proved just what she wanted, and we should be done---but wait, the strategy is indiscreet! After our set of weighings, the judge will know that both the two leftover coins and one of the coins we used in the last weighing are real. Although the lawyer proved that there must be three fake coins, in the process she violated the privacy of three real coins.

The strategies in this section are adapted from Knop's paper \cite{Knop} (in Russian), where he provides different solutions for (100-3-2) problem. Let us suggest another solution.

\begin{strategy}\label{str:3} (80-3-2)

The lawyer divides all coins into four piles: $A$, $B$, $C$, and $D$ with 20 coins each, making sure that piles $A$, $B$, and $C$ all contain one fake coin. She then conducts three weighings comparing each of $A$, $B$, and $C$ to $D$, in the process showing that each of $A$, $B$, and $C$ is lighter than $D$. 
\end{strategy}

Therefore, each of $A$, $B$, and $C$ has to contain a fake coin. Once again, the lawyer successfully proves to the judge that there are 3 fake coins. Unfortunately, the strategy is still indiscreet because the lawyer violates the privacy of all 20 coins in the $D$ pile.

Now we will present the official solution from the competition.

\begin{strategy}\label{str:4} (80-3-2)

The lawyer divides all coins into 5 piles: $A$ and $B$ with 10 coins each, and $C$, $D$, and $E$ with 20 coins each, so that the three fake coins are in piles $A$, $D$, and $E$. The lawyer then conducts three weighings. In the first, she compares $A + C$ against $B + D$, and the scale balances. In the second weighing she compares $A + B$ against $E$, and the scale balances again. In the last weighing she compares $C + D$ against $A + B + E$, and shows that the second pan is lighter.
\end{strategy}

Let us analyze this strategy. The third weighing demonstrates that $A + B + E$ must contain some fake coins. The second weighing shows the judge that the number of fake coins in $A+B+E$ is even; this means that the pile $A+B+E$ has exactly 2 fake coins, one in $E$ and the other in $A+B$. In this case $A+B+C+D$ clearly has fake coins, and because of the first weighing the number of them is even. Therefore, $C+D$ has a fake coin. We can then conclude that the total number of fake coins is 3. What does the judge know about the individual coins? One coin is in $E$, one in $A+B$ and one in $C+D$. In addition, if $A$ contains a fake coin, then $D$ contains the other one. If $B$ contains a fake coin, then $C$ has the other one. The privacy of every individual coin is preserved and the strategy is discreet.

We now offer one more discreet (shapovalov) strategy for this problem.

\begin{strategy}\label{str:5} (80-3-2)

The lawyer divides all the coins into nine piles: $A_1$, $B_1$, $C_1$, $A_2$, $B_2$, $C_2$, $A_3$, $B_3$ and $C_3$ of sizes 24, 1, 2, 24, 1, 2, 23, 2, and 1 respectively. The lawyer demonstrates that $A_1+B_1 = A_2+B_2 = A_3+B_3$. Additionally, she shows that $B_1+C_1= B_2+C_2 = B_3+C_3$.
\end{strategy}

The judge can easily see that if $A_1+B_1$ contains a fake coin, then this pile contains exactly one fake coin and the total number of fake coins is 3. If $A_1+B_1$ doesn't contain a fake coin, then all fake coins must be concentrated in $C_1+C_2+C_3$. Similarly, if $B_1+C_1$ contains a fake coin, then this pile contains exactly one fake coin and the total number of fake coins is 3. If $B_1+C_1$ doesn't contain a fake coin, then all fake coins must be concentrated in $A_1+A_2+A_3$. If both $A_1+B_1$ and $B_1+C_1$ contain a fake coin, then there must be 3 fake coins in the triplet $B_1+B_2+B_3$. Summarizing gives us three different ways for the fake coins to be distributed:

\begin{enumerate}
\item one fake coin in one of each: $A_1$, $A_2$, $A_3$ (sizes 24, 24, 23)
\item one fake coin in one of each: $B_1$, $B_2$, $B_3$ (sizes 1, 1, 2)
\item one fake coin in one of each: $C_1$, $C_2$, $C_3$ (sizes 2, 2, 1)
\end{enumerate}

In all cases, we have ruled out the possibility of there being two fake coins, and no coin in particular has its identity revealed.

See more examples of both insufficient and correct solutions in Knop's article \cite{Knop} (in Russian).

\section{Discreet Coin Weighings}

The original puzzle is tricky, but we've already managed to demonstrate two solutions. Is it always possible to find a solution that respects the privacy of each individual coin? Or, in our new definition, is it always possible to find a discreet set of weighings?

Let us point out the trivial fact that if $f=0$ or $f=t$ (and thus the lawyer has to prove to the judge that all coins are real/fake) the privacy of every coin is guaranteed to be violated as a result of the statement being proven.

In order to prevent the identity of any given coin from being revealed, we will only consider the cases for which $0 < f < t$, and thus the statement that we are trying to prove is itself discreet. However, as the following lemmas prove, it is not always possible to have a discreet set of weighings in this case.

\begin{lemma}\label{thm:f=1}
For $t>1$ and $f=1$ it is impossible to have a discreet strategy.
\end{lemma}

\begin{proof}
Suppose such a strategy exists and the lawyer convinces the judge that the total number of fake coins is 1. Now consider any weighing that is carried out. If it is balanced, then the coins on both pans are all necessarily real. If it is not balanced, then we know that the heavier pan has only real coins. In either case some of the coins are revealed to be genuine, and thus the strategy is indiscreet.
\end{proof}

We use symmetry to prove the following lemma.

\begin{lemma}\label{lem:f=t-1}
For $t>1$ and $f=t-1$ it is impossible to have a discreet strategy.
\end{lemma}

\begin{proof}
Suppose such a strategy exists. By the same logic shown in Lemma~\ref{thm:f=1}, any such strategy will necessarily reveal 1 fake coin. Thus, the strategy is indiscreet.
\end{proof}

\begin{lemma}
For $f=2$ and $d=0$ it is impossible to have a discreet strategy.
\end{lemma}

\begin{proof}
If a weighing is not balanced, then the heavier pan must have only real coins. If all the weighings are balanced, the judge can't differentiate between 2 fake coins and 0 fake coins---we end up proving nothing.
\end{proof}

We will later show that the above examples are not the only exceptions to the existence of a discreet strategy, but these next examples are more involved; they are thus produced in Section~\ref{oddity}.

\section{The Revealing Factor and Coefficient}

Let's go all the way back to Strategy~\ref{str:1} (80-2-1), in which the lawyer splits the coins into two groups of 40. Suppose the judge just knows that there are 2 fake coins. What are his chances of finding a single fake coin before the weighing? They are simply 2 in 80. After the weighing the judge knows that there is a fake coin in each group of 40, so his chance of finding one coin is 1 in 40---the same as before. It seems as though no information is revealed, but this is not the case.

It turns out that some information is revealed in the process of weighing the coins. Before the weighings, the two counterfeits can be any of the 80 coins, and the number of equally likely distributions of these fake coins is $\binom{80}{2} = 3160$. After the weighings, there is one fake in each pile of 40, and the number of possibilities is reduced to $\binom{40}{1}^2 =1600$.

We would like to introduce the notions of a revealing factor and a revealing coefficient to quantify this observation. Before the weighings, if the judge knows that the number of fake coins is exactly $f$, then any set of $f$ coins might be the set of fake coins. The number of equally likely possibilities is $\binom{t}{f}$, and we will call this value \textit{old possibilities}. After the weighings, the set of possibilities decreases so that not any arbitrary group of $f$ coins could be the set of fake coins. We call the number of sets of $f$ coins that could be fake after the weighings are done the \textit{new possibilities}.

\begin{definition}
We call the ratio of the number of old possibilities to the new possibilities after a successful strategy the \textit{revealing factor}. We denote it by $X$:
$$X = \frac{\text{\# old possibilities}}{\text{\# new possibilities}}.$$
\end{definition}

We would also like to introduce the notion of a \textit{revealing coefficient} as used in \cite{TK}. The revealing coefficient is the portion of information that is revealed in the process of proving that there are exactly $f$ fake coins. The revealing coefficient is close to 1 if the judge knows the exact location of each fake coin and 0 if he receives no extra information other than that which was intended: the number of fake coins.

\begin{definition}
The revealing coefficient is defined as $1 - 1/X$. We denote it by $R$:
$$R = 1 - \frac{\text{\# new possibilities}}{\text{\# old possibilities}}.$$
\end{definition}

We would like both the revealing coefficient and the revealing factor to be as small as possible in order to minimize the transfer of information.

In Strategy~\ref{str:1}, the revealing factor is $X=3160/1600=1.98.$ The revealing coefficient is $R=(3160-1600)/3160 \approx 0.494$, slightly less than one half.

Let's calculate the revealing factor and coefficient for our first discreet solution (Strategy~\ref{str:4}) to the original problem to help solidify these two new concepts. We have five piles $A$, $B$, $C$, $D$, and $E$ of 10, 10, 20, 20 and 20 coins correspondingly. We showed that there are two possibilities: either piles $A$, $D$ and $E$ contain one fake coin each, or piles $B$, $C$, and $E$ contain one fake coin each. After the weighings the number of new possibilities is $10 \cdot 20 \cdot 20 + 10 \cdot 20 \cdot 20= 8,000$. The number of old possibilities is $\binom{80}{3}= 82,160$. The revealing factor is $X=\frac{82160}{8000}=10.27$ and the revealing coefficient is $R = 1 - \frac{8000}{82160} \approx 0.903$. 

We will give one lemma regarding the revealing coefficient:

\begin{lemma}\label{0<R<1}
After the first weighing with an equal number of coins in both pans, $0<R<1$.
\end{lemma}

\begin{proof}
The right side of this inequality is trivial as it is always true that the number of new possibilities is greater than 0. The left side of this inequality is equivalent to saying that the first weighing necessarily reveals information. Consider any weighing using $2n$ coins from the original pile of $t$, where we have $n$ coins in each pan. If the pans are balanced, we know that both groups of $n$ coins have the same number of fake coins. If one of the pans is lighter than the other, then we know that there are more fake coins in that pan than there are in the other. In either case, it is no longer possible for the $f$ fake coins to be distributed any way we like; they must be distributed in a way that is consistent with the weighing. As a result, it is always true that the number of new possibilities is less than $\binom{t}{f}$, or the number of old possibilities.
\end{proof}

\subsection{Different Strategies for Given Parameters Reveal Different Amounts of Information}

Suppose we have 80 coins and we want to prove that 4 are fake as opposed to 3. We can use Strategy~\ref{str:1*} to do that: Divide all coins into four piles of 20 with each pile containing one fake coin. Showing that all piles weigh the same tells the judge the number of fake coins must be a multiple of 4, and we are done. We can, however, produce another discreet strategy (as hinted at briefly in Section~\ref{simp}): Simply divide the coins into two piles of 40 and put two fake coins in each pile. After comparing the two piles the judge knows that the number of fake coins is even. Both strategies are discreet, but the revealing factor and coefficient are different. The total number of possibilities before weighings is $\binom{80}{4} = 1,581,580$. After the first strategy the number of new possibilities is $20^4=160,000$. The revealing factor is $X \approx 9.9 $ and the revealing coefficient $R \approx 0.899$. After the second strategy the number of new possibilities is $\binom{40}{2}^2=608,400$. The revealing factor is $X \approx 2.60$ and the coefficient is $R \approx 0.615$.

The second strategy is significantly less revealing; dividing all the coins into fewer equivalent piles is clearly preferable in this case. 

We can generalize these two strategies for the case when $f$ and $t$ have a common divisor greater than 1. If $a>1$ divides both $f$ and $t$ but not $d$, then the following strategy is discreet:

\begin{strategy}\label{str:6}
The lawyer divides all the coins into $a$ piles, each having the same number of fake coins. She then demonstrates to the judge that the piles all have the same weight, thus proving that the number of fake coins is divisible by $a$.
\end{strategy}

The revealing factor for this strategy is
$$X =\frac{\binom{t}{f}}{\binom{t/a}{f/a}^a} \sim \dfrac{f^{f}}{f!} \bigg( \dfrac{(\frac{f}{a})!}{(\frac{f}{a})^{(\frac{f}{a})}} \bigg) ^a, $$ where the right hand side is the value that $X$ approaches as $t$ tends to infinity.

There are often many values of $a$ that satisfy the above conditions, but the smallest possible such value will reveal the least amount of information for this strategy.  Notice that, when $f$ divides $t$ and we choose $a=f$, the weighing scheme is identical to that of Strategy~\ref{str:1*}.

\subsection{The Revealing Factor/Coefficient for Indiscreet Weighings}

The revealing factor/coefficient can be defined for both discreet and indiscreet strategies. Surprisingly, we often see that indiscreet strategies reveal less information than discreet strategies do. Let's go back to the original problem and its ``wrong,'' or rather indiscreet, solutions. In the first solution, described in Strategy~\ref{str:2}, after the weighings the judge knows the locations of 3 real coins. The other coins are divided into 3 groups of 26, 26 and 25 coins, containing one fake coin each. Thus the number of possibilities after the weighings is $26\cdot 26 \cdot 25 = 16,900$. The number of possibilities before the weighings is $\binom{80}{3}= 82,160$, so the revealing factor is $X \approx 4.86$, and the revealing coefficient is $R \approx 0.794$. We see that although this strategy is indiscreet and the privacy of 3  real coins is violated, it is less revealing than the discreet Strategy~\ref{str:4} with a revealing factor of 10.27. We can see that the three coins that were exposed as real effectively ``sacrificed'' their privacy in order to make the other coins more secure in their wish to remain hidden.

The next indiscreet example, Strategy~\ref{str:3}, is more revealing than Strategy~\ref{str:2} as the groups containing the fake coins are smaller in comparison. The number of new possibilities is $20^3=8000$, so the revealing factor is $X \approx 10.27$ and the revealing coefficient is $R \approx 0.903$; interestingly, these values are exactly the same as those given by discreet Strategy~\ref{str:4} even though one fourth of the entire set of coins has its authenticity revealed.

\subsection{An Optimality Proof}

Here we would like to show a proof of optimality for a given discreet strategy. Namely, we will consider the case when $t=2k$, $f=2$ and $d=1$, for some positive integer $k > 1$. We've already discussed the strategy (see Strategy~\ref{str:1*}) using one weighing: divide all the coins into two piles of size $k$, put fake coins into the separate piles, and compare the piles. 

It might seem obvious that this is a ``least revealing,'' or optimal, strategy, but still we need a proof. First, we will introduce more definitions and notation. During any weighing, a coin's presence on the \textbf{l}eft pan is denoted by $L$, a coin's presence on the \textbf{r}ight pan is denoted by $R$, and a coin not participating (one that is left \textbf{o}utside of the weighing) is denoted by $O$.

After all the weighings, every coin's path can be described as a string of $L$'s, $R$'s, and $O$'s.

\begin{definition}
The string of $L$'s, $R$'s, and $O$'s corresponding to the location of a given coin in every weighing is called the coin's \textit{itinerary}.
\end{definition}

Given an itinerary $\delta$, we denote the set of all coins with this itinerary as $\delta$, and the size of this set as $|\delta|$. We will introduce an involutive operation on itineraries called conjugation:

\begin{definition}
Given an itinerary $\delta$, \textit{the conjugate itinerary}, denoted by $\bar{\delta}$, is the unique itinerary in which all $R$'s are replaced by $L$'s, and all $L$'s are replaced by $R$'s.
\end{definition}

Notice that this is an involution as $\bar{\bar{\delta}} = \delta$. In addition, the only self-conjugate itinerary is a string of $O$s. After the weighings we can partition all the coins into groups by their itineraries.

In the following preliminary lemma it is not necessary that $t$ be even.

\begin{lemma}
If $f=2$ and $d=1$, then the set of itineraries of a  discreet strategy satisfies the following properties: If there are coins with itinerary $\delta$, then there are coins with itinerary $\bar{\delta}$. In addition, there are no coins with a self-conjugate itinerary. Also, all weighings must balance. 
\end{lemma}

\begin{proof}
All weighings must be balanced, otherwise the heavier pan must contain only real coins and the strategy is indiscreet.

It follows that the two fake coins can't be in the same pan at any point. Moreover, if a fake coin is in one of the pans during a weighing, the other fake coin must be in the other pan in order to balance it.  This means that the fake coins have conjugate itineraries.

Because of the above condition, if there exists a coin with an itinerary $\delta$ and there are no coins with itinerary $\bar{\delta}$, then all the coins in $\delta$ are necessarily real.

If one coin did not partake in any of the measurements and all the weighings were balanced, then we have not disproven the possibility of only one fake coin. Thus, there may
not be any coins with a self-conjugate itinerary.
\end{proof}

Now we are ready to prove the optimality theorem for an even number of coins.

\begin{theorem}\label{alg-opt}
If $t$ is even, $f=2$ and $d=1$, then Strategy~\ref{str:1} is the least revealing out of any possible strategy.
\end{theorem}

\begin{proof}
All the coins are partitioned into groups with the same itineraries and all the itineraries are paired up by conjugation. The set of itineraries is $(\delta_j,\bar{\delta_j})$ for $j=1,2,3,\ldots$. If one fake coin belongs to $\delta_j$, then the other fake coin must belong to $\bar{\delta_j}$. This means the total number of new possibilities is
$$\sum_j |\delta_j|\cdot |\bar{\delta_j}|,$$
where $\Sigma_j (|\delta_j|+|\bar{\delta_j}|)=t$ is a fixed number.

Standard algebra arguments show that the number of new possibilities is maximized when there is exactly one itinerary pair with two itineraries of equal size.
\end{proof}

\subsection{An Oddity: The Case of Odd $t$}\label{oddity}

We would like to discuss optimality in a more complicated case: when the total number of coins is odd, $f=2$, and $d=1$.

\begin{lemma}
A discreet strategy for the given parameters must generate at least 6 different itineraries. Among the itineraries there will be at least 3 conjugate pairs; additionally, the number of coins in the two groups in each pair must be of different parity.
\end{lemma}

\begin{proof}
Since the number of coins is odd, every weighing has some coins that aren't on the scale. This means that one weighing is not enough and  we need at least one more weighing to account for the coins that are left out. Let us restrict the itinerary to these two weighings. We have coins with itineraries $LO$, $RO$, $OL$, and $OR$. Since the number of coins not counted by these itineraries is odd, there exist coins with yet another itinerary. Furthermore, as the total number of itineraries is even, we can conclude there are at least 6 of them.

Since the total number of coins is odd, all pairs of itineraries can't have numbers of coins with the same parity; at least one pair must have different parity. This pair is on the scale at some point with the same number of coins on both pans, implying that at least one other pair of itineraries has different parities. The total number of pairs of itineraries with different parities must be odd, hence there are at least 3 of them.
\end{proof}

Now we will provide some more promised examples for which a discreet strategy does not exist.

\begin{corollary}\label{thm:5-2-1-impossible}
For $f=2$ and $d=1$, it is impossible to have a discreet strategy when $t=3$, $t=5$, or $t=7$.
\end{corollary}

\begin{proof}
We already know that the discreet strategy is impossible for $t=3$ because $f=t-1$; see Lemma~\ref{lem:f=t-1}.

Both three and five coins clearly can't have 6 different itineraries. In addition, three of the pairs of itineraries have different parity, which means there must be at least 3 coins in each pair. The total number of coins is at least 9.
\end{proof}

What happens when $t > 7$? Does the proof for the nonexistence of a discreet strategy become increasingly difficult as $t$ increases? No---it turns out that this pattern does not continue. Here is a discreet strategy that works for $t>7$; this strategy is a modification of Strategy~\ref{str:5}.

\begin{customstr}{5*}\label{str:5*} (($2k+1$)-2-1)
Consider piles $A$ and $\bar{A}$ of sizes $\Bigl\lfloor\dfrac{t}{2}\Bigr\rfloor-2$ and $\Bigl\lfloor\dfrac{t}{2}\Bigr\rfloor-3$ respectively. Piles $B$ and $\bar{B}$ have sizes 1 and 2 respectively. Piles $C$ and $\bar{C}$ have sizes 2 and 1. The lawyer distributes the two fake coins into piles $A$ and $\bar{A}$, $B$ and $\bar{B}$, or $C$ and $\bar{C}$. 

This strategy has two weighings. In the first the lawyer compares $\Bigl\lfloor\dfrac{t}{2}\Bigr\rfloor-1$ coins belonging to $A$ and $B$ against $\Bigl\lfloor\dfrac{t}{2}\Bigr\rfloor-1$ coins belonging to $\bar{A}$ and $\bar{B}$. In the second weighing she compares three coins belonging to $B$ and $C$ against the same number of coins in $\bar{B}$ and $\bar{C}$.
\end{customstr}

\begin{lemma}\label{thm:9-2-1-possible}
For $f=2$ and $d=1$, Strategy~\ref{str:5*} is discreet when $t$ is odd and $t>7$.
\end{lemma}

\begin{proof}
All coins were on the scale and all the weighings were balanced. This means that there are two fake coins.

No information about any particular coin is revealed as the fake coins can belong to any pair of groups with conjugate itineraries.
\end{proof}

\begin{lemma}
If $t$ is odd, $f=2$ and $d=1$, then Strategy~\ref{str:5*} is least revealing.
\end{lemma}

\begin{proof}
Similar to the proof of Lemma~\ref{alg-opt}, to increase the number of new possibilities we would like to move as many coins as possible to one of the conjugate pairs. Given that we have to keep at least 3 pairs of itineraries of different parity, the distribution for $(|A|,|\bar{A}|)$, $(|B|,|\bar{B}|)$, and $(|C|,|\bar{C}|)$ must be $\Big( \Bigl\lfloor\dfrac{t}{2}\Bigr\rfloor-2, \Bigl\lfloor\dfrac{t}{2}\Bigr\rfloor-3 \Big)$, $(1,2)$, and $(2,1)$, respectively. The number of new possibilities is 

$$|A|\cdot |\bar{A}| + |B|\cdot |\bar{B}| + |C|\cdot |\bar{C}| = \Big( \Bigl\lfloor\dfrac{t}{2}\Bigr\rfloor-2 \Big) \Big( \Bigl\lfloor\dfrac{t}{2}\Bigr\rfloor-3 \Big)+1\cdot 2+2\cdot 1.$$
\end{proof}

As the Theorem~\ref{discreetmethod} shows, for large enough $t$ we can often find a discreet strategy. Let us first introduce such a strategy that is defined for $\Bigl\lfloor\dfrac{t}{f}\Bigr\rfloor \ge 4$ and $f \nmid d$. The strategy is a generalization of Strategy~\ref{str:5*}.

\begin{customstr}{5**}\label{str:5**}
Let's say that we have $t = fk+r$, where $k$ and $r$ are positive integers, $0 < r < f$, and $k \ge 4$.
We will begin by splitting the coins into $3f$ total groups: $A_1, A_2, \ldots, A_f$, $B_1, B_2,\ldots, B_f$, and $C_1, C_2, \ldots, C_f$ with the same itineraries. The lawyer will put one fake coin in each of either $A_i$, $B_i$, or $C_i$, for $i=1,2,\ldots,f$. 

In groups $A_i$ for $1 \leq i \leq r$, we will have $|A_i|=k-2$, and for $r+ 1 \leq i \leq f$ we will have $|A_i|=k-3$. Similarly, we will have $|B_i|=1$ for $1 \leq i \leq r$ and $|B_i|=2$ otherwise, and in groups $C_i$ for $1 \leq i \leq r$ we will have $|C_i|=2$ and $|C_i|=1$ for all other values of $i$.

Now we carry out the weighings as follows. In $f-1$ weighings, we show that the $k-1$ coins from each $A_i+B_i$ for $1 \leq i \leq f$ balance one another on the scale. In $f-1$ more weighings, we demonstrate that the $3$ coins from each $B_i+C_i$ for $1 \leq i \leq f$ are equal in weight.
\end{customstr}

\begin{theorem}\label{discreetmethod}
If $\Bigl\lfloor\dfrac{t}{f}\Bigr\rfloor \ge 4$ and $0< d < f$, Strategy~\ref{str:5**} is discreet.
\end{theorem}

\begin{proof}
What can the judge conclude? Every coin was on the scale at some point. Suppose that a fake coin is in one of the groups. There are $f-1$ other groups of the same weight, so there must be at least $f$ fake coins. As $d<f$, there are exactly $f$ fake coins.

Using similar arguments as in the analysis of Strategy~\ref{str:5} the judge can conclude that there are three mutually-exclusive possibilities: 
\begin{itemize}
\item each $A$ group contains one fake coin,
\item each $B$ group contains one fake coin,
\item each $C$ group contains one fake coin,
\end{itemize}

The privacy of every coin is respected.
\end{proof}

\section{One Shot: Guessing a Single Fake Coin}\label{sec:6}

So far we have only looked at how much information is revealed about the group of fake coins as a whole. We might consider another goal of the judge: trying to guess the location of only one fake coin \cite{Zhang}. We already mentioned that Strategy~\ref{str:1} (80-2-1) keeps the judge's ability to guess the fake coin at $1/40$, the same as before the weighings.

Strategy~\ref{str:2} (80-3-2) reveals that each of three piles of sizes 26, 26, and 25 contains a fake coin. The best way for the judge to guess is to choose one coin out of the pile with 25 coins. The probability of success is $1/25$, which is slightly higher than before the weighings: $3/80$.

Strategy~\ref{str:3} (80-3-2) reveals that each of three piles of size 20 contains a fake coin. The probability of guessing one fake coin is $1/20$. This is worse than Strategy~\ref{str:2}, as the judge is more likely to find a fake coin.

Let us consider the official solution (Strategy~\ref{str:4}) for the original problem. At the end the judge knows that there are three coins that are distributed as follows: 

\begin{enumerate}
\item one fake coin in $E$ of size 20,
\item one fake coin in $A+B$ of size 20,
\item one fake coin in $C+D$ of size 40.
\end{enumerate}

We see that if the judge wants to guess one fake coin, he can do so with probability 1/20. This strategy also improved the judge's chances of guessing a fake coin.

Strategy~\ref{str:5} (80-3-2) can be analyzed in a manner similar to the previous strategies. Recall that this strategy leaves us with three distinct cases for the distribution of the three coins:

\begin{enumerate}
\item one fake coin in one of each: $A_1$, $A_2$, $A_3$ (sizes 24, 24, 23)
\item one fake coin in one of each: $B_1$, $B_2$, $B_3$ (sizes 1, 1, 2)
\item one fake coin in one of each: $C_1$, $C_2$, $C_3$ (sizes 2, 2, 1)
\end{enumerate}

If the judge wants to guess one fake coin, he can pick randomly from $A_3 + B_1 + C_3$, a pile of size 25. Thus, his chances of guessing one fake coin are 1/25: better than before the weighings. However, the judge's chances might significantly increase depending on how the lawyer distributes the fake coins. This is discussed further in the next section.

\subsection{Minimizing the Chance of Guessing One Fake Coin}

Let us for the duration of this subsection slightly alter the lawyer's goal. Given $t$, $f$, and $d$ we want the lawyer to create a strategy that proves to the judge that the number of fake coins is $f$ as supposed to $d$. Just as before, the weighings are allowed to be indiscreet, but with the additional goal of minimizing the judge's probability of guessing a single fake coin.

Consider the lawyer's options when it comes to hiding the fake coins in Strategy 5. A huge problem arises if she chooses each one of the three cases ($A$, $B$, or $C$) with probability 1/3. If the judge knows that, then he can pick either $B_1$, $B_2$, or $C_3$, and find a fake coin with probability 1/3. To make matters worse, if the judge picks a coin from each of $B_1$, $B_2$, and $B_3$, then he can find all of the fake coins with probability 1/6.

This strategy was clearly too unbalanced. Suppose the lawyer decides to place the fake coins in groups $A_1$, $A_2$, $A_3$ with probability $p_A$, in groups $B_1$, $B_2$, $B_3$ with probability $p_B$, and in groups $C_1$, $C_2$, $C_3$ with probability $p_C=1-p_A-p_B$.

Now suppose the judge is trying to randomly guess one counterfeit coin. If the judge guesses from the group $A_1+A_2+A_3$ he is successful with probability $p_A/23$. This is because the smallest such $|A_i|$ that he can guess from is $23$. If he guesses from the group $B_1+B_2+B_3$, he is successful with probability $p_B/1$, and if he guesses from the group $C_1+C_2+C_3$ he is successful with probability $p_C/1$; this is again because the smallest such $|B_i|$ and $|C_i|$ that the judge can pick from are both $1$. If the lawyer wants to decreases the chances of the judge finding one fake coin, she needs to pick probabilities proportional to the sizes of the piles. Namely: $p_A=23/25$ and $p_B = p_C = 1/25$. With this strategy, the judge will be able to guess one coin with almost the same probability after the weighings as before the weighings: $1/25$, which is only slightly larger than $f/t=3/80$.

The initial probability of guessing a fake coin is $f/t$ if the judge knows that there are exactly $f$ fake coins. After the weighings, this probability can only stay the same or increase. This minimum can be achieved if $t$ and $f$ have a common divisor $a$ that doesn't divide $d$, utilizing Strategy~\ref{str:6}.

We proved in Lemma~\ref{0<R<1} that any set of weighings will increases the judge's chances of finding all the fake coins. On the other hand, if the judge only wishes to find one fake coin, the lawyer can sometimes use a strategy that doesn't improve the judge's chances. However, whether the lawyer's goal is to hide just one fake coin or the entire set of fake coins, indiscreet strategies always help the judge:

\begin{lemma}
Any indiscreet strategy increases the judge's ability to guess a single fake coin.
\end{lemma}

\begin{proof}
At least one of the coins will be revealed to be genuine or fake. This increases the judge's ability to guess a fake coin.
\end{proof}

On the other hand, if $t$ is large, $f$ and $d$ are small, and $f$ doesn't divide $d$, there exists an indiscreet strategy that proves that there must be $f$ fake coins, and the judge's ability to guess one fake coin is very close to $f/t$. This strategy is a generalization of Strategy~\ref{str:2} (80-3-2):

\begin{customstr}{2*}\label{str:2*}
The lawyer divides all coins into $f$ piles containing one fake coin each, along with some leftover coins so that the number of leftover coins is less than $f$. The lawyer shows that each large pile weighs the same. After that she compares the leftover coins to each other and to some other coins if necessary to demonstrate that there are at least $d+1$ coins that weigh the same and this set of coins contains the leftover coins. 
\end{customstr}

This shows that the leftover coins have to be real. This also proves that all the fake coins are in the $f$ large piles, with each pile containing exactly one fake coin. 

For this procedure, the judge borrows at most $d+1-f$ coins from all of the larger piles and shows that they are genuine. This leaves piles of size at least $\lfloor \frac{t}{f} \rfloor - \lceil \frac{d}{f} \rceil$ that contain at least one fake coin. This means that for $f$ not dividing $d$, we can create a strategy that gives the judge a probability of $ \dfrac{1}{\lfloor \frac{t}{f} \rfloor - \lceil \frac{d}{f} \rceil} $ of guessing a single fake coin. The probability of guessing one fake coin with this strategy clearly approaches $f/t$ as $t$ grows larger.

Note that in this strategy the lawyer reveals $d+1$ coins to be real. Depending on the values of $t$, $f$, and $d$, the lawyer can sometimes tweak the strategy to reveal the privacy of fewer coins.

\section{Conclusion}

Do fake coins really need a lawyer's protection in the courtroom? Most likely not. But mathematicians make a living by reducing difficult problems to easier, more manageable ones. In short, our discussion demonstrates that collecting aggregated information from databases reveals some additional information in the process.

As an illuminating example, let's say you filled in an anonymous survey about your taste in movies. Although your opinions are very unpopular, you feel safe because you never once mentioned your name. The surveyors publish the results and mention the curious fact that people who live in towns starting with W all hate \textit{Star Trek}; now you're in real trouble. Your friends, your spouse, and even your dog know that you filled in the survey and now realize that you have been lying to them for all these years. 

Although the above example doesn't use coins, we have conclusively shown that aggregated data collection decreases your privacy; this paper is our attempt to quantify by how much.

\section{Acknowledgments}

We are grateful to the MIT-PRIMES program for allowing us to conduct this research.

\end{document}